\documentclass[11 pt,twoside,a4paper]{amsart}
\usepackage{amssymb}
\usepackage{amsmath,amscd}
\usepackage[initials,nobysame,shortalphabetic]{amsrefs}
\usepackage{verbatim}

\def\dbar{\bar\partial}

\def\C{{\mathbb C}}

\def\Re{{\rm Re\,  }}

\newcommand{\Com}[1]{}

\def\be{\begin{equation}}
\def\ee{\end{equation}}

\newtheorem{thm}{Theorem}[section]
\newtheorem{lma}[thm]{Lemma}
\newtheorem{cor}[thm]{Corollary}
\newtheorem{prop}[thm]{Proposition}

\theoremstyle{definition}

\newtheorem{df}{Definition}

\theoremstyle{remark}

\newtheorem{preremark}{Remark}
\newtheorem{preex}{Example}

\newenvironment{remark}{\begin{preremark}}{\end{preremark}}

\numberwithin{equation}{section}

\ifdefined\custompagenr
\fi

\begin{document}

\title[Singular hermitian metrics on holomorphic vector bundles]{Singular hermitian metrics on holomorphic vector bundles}


\author{Hossein Raufi}

\address{H. Raufi\\Department of Mathematics\\Chalmers University of Technology and the University of Gothenburg\\412 96 G\"OTEBORG\\SWEDEN}

\email{raufi@chalmers.se}




\begin{abstract}
We introduce and study the notion of singular hermitian metrics on holomorphic vector bundles, following Berndtsson and P{\u{a}}un. We define what it means for such a metric to be positively and negatively curved in the sense of Griffiths and investigate the assumptions needed in order to locally define the curvature $\Theta^h$ as a matrix of currents. We then proceed to show that such metrics can be regularised in such a way that the corresponding curvature tensors converge weakly to $\Theta^h$. Finally we define what it means for $h$ to be strictly negatively curved in the sense of Nakano and show that it is possible to regularise such metrics with a sequence of smooth, strictly Nakano negative metrics.
\end{abstract}

\maketitle

\section{Introduction}
\noindent Let $E\to X$ be a holomorphic vector bundle over a complex manifold $X$ and let $h$ be a hermitian metric on $E$. In applying the methods of differential geometry to the study of $(E,X)$ the connection matrix and curvature associated with $h$ play a major role. In the classical setting these constructions assume that $h$ is smooth as a function from $X$ to the space of non-negative hermitian forms on the fibers.

However in \cite{D2} Demailly introduced the notion of \textit{singular} hermitian metrics for line bundles, and in a series of papers he and others proceeded to investigate these and prove that, generally speaking, they are a fundamental tool in interpreting notions of complex algebraic geometry analytically.

In \cite{C} and \cite{BP} two different notions of singular hermitian metrics on a holomorphic vector bundle are introduced. We will adopt the latter definition which is the following (see section 3 for a comparison and discussion of the quite different definition used \cite{C}):

\begin{df}\label{df:1}
Let $E\to X$ be a holomorphic vector bundle over a complex manifold $X$. A singular hermitian metric $h$ on $E$ is a measurable map from the base space $X$ to the space of non-negative hermitian forms on the fibers.
\end{df}

On a holomorphic line bundle a hermitian metric $h$ is just a scalar-valued function so that $\theta=h^{-1}\partial h=\partial\log h$, $\Theta=\dbar\partial\log h$ and hence these objects are well-defined as currents as long as $\log h\in L^1_{loc}(X)$. But for holomorphic vector bundles with rank$E\geq2$ it is not clear what the appropriate notions of connection and curvature associated with $h$ are.

Although it is not immediate how to make sense of the curvature of a singular hermitian metric $h$, it is nevertheless still possible to define what it means for $h$ to be positively and negatively curved in the sense of Griffiths. The definition that we will adopt is the following:

\begin{df}\label{df:2}
Let $E\to X$ be a holomorphic vector bundle over a complex manifold $X$ and let $h$ be a singular hermitian metric. We say that $h$ is negatively curved in the sense of Griffiths if $\|u\|^2_h$ is plurisubharmonic for any holomorphic section $u$. Furthermore we say that $h$ is positively curved in the sense of Griffiths if the dual metric is negatively curved.
\end{df}

This is a very natural definition as these conditions both are well-known equivalent properties for smooth metrics; see section 2 where these facts are reviewed. In fact this is almost identical to the definition adopted by Berndtsson-P{\u{a}}un (\cite{BP} Definition 3.1), except that they require $\log\|u\|^2_h$ to be plurisubharmonic, but one can show that the two definitions are equivalent; see section 2.

The main question that we are concerned with in this paper is:\vspace{0.4cm}\\\textit{Given these two definitions, is it possible to define $\theta^h$ and in particular $\Theta^h$, in a meaningful way; for example as currents with measure coefficients?}\vspace{0.4cm}\\Furthermore, if we want these concepts to be useful in practice, one should also be able to approximate them in an appropriate sense.

Of course, Definition \ref{df:1} is by itself far too liberal to provide us with any answer to this question. However, it turns out that the additional requirements in Definition \ref{df:2} rule out most of the possible pathological behaviour.
The following properties are more or less immediate consequences of Definition \ref{df:2}.

\begin{prop}\label{prop:basic}
Let $h$ be a singular hermitian metric on a holomorphic vector bundle $E$, and assume that $h$ is negatively curved in the sense of Griffiths as in Definition \ref{df:2}.\\
(i) (\cite{BP}, Proposition 3.1) If $E$ is a trivial vector bundle over a polydisc, there exists a sequence of smooth hermitian metrics $\{h_\nu\}$ with negative Griffiths curvature, decreasing to $h$ pointwise on any smaller polydisc.\\
(ii) $\log\det h$ is a plurisubharmonic function. In particular, if $\det h\not\equiv0$, then $\log\det h\in L^1_{loc}(X)$.
\end{prop}

In (i), the approximating sequence is obtained through the well known technique of convolution with an approximate identity (see section 6). Also, by duality, if $h$ is positively curved as in Definition \ref{df:2}, there exists an increasing regularising sequence.

Now it turns out that it is possible to define the connection matrix of a singular hermitian metric that is positively or negatively curved as in Definition \ref{df:2}.

\begin{prop}\label{prop:conn}
Let $h$ be a singular hermitian metric on a holomorphic vector bundle $E$, that is negatively curved in the sense of Griffiths, as in Definition \ref{df:2}. Then the current $\partial h$ is locally an $L^2$-valued form and $\theta^h:=h^{-1}\partial h$ is an a.e. well-defined matrix of $(1,0)-$forms.
\end{prop}

For the curvature, however, the situation turns out to be more involved.

\begin{thm}\label{thm:counter} 
Let $\Delta\subset\mathbb{C}$ denote the unit disc and let $E=\Delta\times\mathbb{C}^2$ be the trivial vector bundle over $\Delta$. Let $h$ be the singular hermitian metric,
\begin{displaymath}
h=\left(\begin{array}{cc}
1+|z|^2 & z \\
\bar{z} & |z|^2
\end{array} \right).
\end{displaymath}
Then, $h$ is negatively curved in the sense of Griffiths, as in Definition \ref{df:2}, and the connection matrix of $h$ is given by (see section 3),
\begin{displaymath}
\theta^h:=h^{-1}\partial h=\left(\begin{array}{cc}
\frac{1}{z} & 0 \\
-\frac{1}{z^2} & \frac{1}{z}
\end{array} \right)dz.
\end{displaymath}
Hence, $\theta^h$ is not locally integrable on $\Delta$, and $\Theta^h:=\dbar\theta^h$ is not a current with measure coefficients.
\end{thm}

This theorem shows that it is not possible to define the curvature in general, just using Definition \ref{df:2}. The existence of examples such as the metric in Theorem \ref{thm:counter} forces us to disregard the singular part of the metric, which is characterised by the fact that $\det h$ vanishes there. If we impose the additional condition that $\det h>\varepsilon$, for some $\varepsilon>0$, we get the following theorem. (Here and in what follows $\tilde{\Theta}^{h_\nu}:=i\Theta^{h_\nu}(\xi,\xi)$, where $\xi$ denotes an arbitrary smooth vector field.)

\begin{thm}\label{thm:main}
Let $E\to X$ be a holomorphic vector bundle over a complex manifold $X$, and let $h$ be a singular hermitian metric on $E$ that is negatively curved in the sense of Griffiths, as in Definition \ref{df:2}. Let furthermore $\{h_\nu\}$ be any approximating sequence of smooth, hermitian metrics with negative Griffiths curvature, decreasing pointwise to $h$. 

If there exists $\varepsilon>0$ such that $\det h>\varepsilon$, then $\Theta^h:=\dbar\theta^h$ is well-defined and $\tilde{\Theta}^{h_\nu}$ converge weakly to $\tilde{\Theta}^h$ as currents with measurable coefficients.
\end{thm}

If for some reason it is known that $\tilde{\Theta}^{h_\nu}\in L_{loc}^1(X)$ uniformly in $\nu$, then Theorem \ref{thm:main} holds without using the assumption $\det h>\varepsilon$. (A trivial example of this is when the vector bundle is a direct sum of line bundles, i.e. the singular hermitian metric is diagonal.)

As we will soon discuss, despite the rather unpleasant condition $\det h>\varepsilon$, the setting of Theorem \ref{thm:main} is sufficient to prove vanishing theorems for these types of singular vector bundles. Before we turn to this, however, we first note that the following corollary is an immediate consequence of Theorem \ref{thm:main}.

\begin{cor}\label{cor:1}
Assume that $F:=\{\det h=0\}$ is a closed set and assume furthermore that there exists an exhaustion of open sets $\{U_j\}$ of $F^c$ such that $\det h>1/j$ on $U_j$. Then $\Theta^h$ exists as a current on $F^c$.
\end{cor}

From now on we will assume that the assumptions of this corollary are met.

Now for line bundles one of the main theorems concerning singular metrics is the Demailly-Nadel vanishing theorem \cite{D1},\cite{N}. One statement of this theorem is that for forms of appropriate bidegrees with values in a line bundle, it is possible to solve the inhomogenous $\dbar-$equation on any complete K\"ahler manifold, if the metric is strictly positively curved as a current. To prove a corresponding result in the vector bundle case requires an appropriate notion of being curved in the sense of Nakano in the singular setting.

Just as for curvature in the sense of Griffiths, there exists an alternative characterisation of Nakano negativity that we can use in the singular setting (\cite{B} section 2). Namely let $u=(u_1,\ldots,u_n)$ denote an $n-$tuple of holomorphic sections of $E$, and define $T_u^h$, an $(n-1,n-1)-$form through
\be\label{eq:naka}
T_u^h=\sum_{j,k=1}^n(u_j,u_k)_h\widehat{dz_j\wedge d\bar{z}_k}
\ee
where $(z_1,\ldots,z_n)$ are local coordinates and $\widehat{dz_j\wedge d\bar{z}_k}$ denotes the wedge product of all $dz_i$ and $d\bar{z}_i$ except $dz_j$ and $d\bar{z}_k$, multiplied by a constant of absolute value 1, chosen so that $T_u^h$ is a positive form. Then a short computation yields that $h$ is negatively curved in the sense of Nakano if and only if $T_u^h$ is plurisubharmonic in the sense that $i\partial\dbar T_u^h\geq0$; see section 2. Choosing $u_j=u_k=u$ we see that if this requirement is met, $h$ is negatively curved in the sense of Griffiths, as in Definition \ref{df:2}, as well. Hence $\Theta^h$ is well-defined on $F^c$.  

We will adopt the following definition of being strictly negatively curved in the sense of Nakano.

\begin{df}\label{df:3}
We say that a singular hermitian metric $h$ on a vector bundle $E$ is strictly negatively curved in the sense of Nakano if:\\
(i) the $(n-1,n-1)-$form $T_u^h$ given in (\ref{eq:naka}) is plurisubharmonic for any $n-$tuple of holomorphic sections $u=(u_1,\ldots,u_n)$ .\\
(ii) there exists $\delta>0$ such that on $F^c$
\be\label{eq:str_naka}
\sum_{j,k=1}^n\big(\Theta_{jk}^hs_j,s_k\big)_h\leq-\delta\sum_{j=1}^n\|s_j\|^2_h,
\ee
for any $n-$tuple of sections $s=(s_1,\ldots,s_n)$. 
\end{df}

Here $\{\Theta_{jk}\}$ are the matrix-valued distributions defined through
$$\Theta^h=\sum_{j,k=1}^n\Theta_{jk}^hdz_j\wedge d\bar{z}_k$$
and so the expression in (\ref{eq:str_naka}) should be interpreted in the sense of distributions.

In section 5 we prove the following approximation result.

\begin{prop}\label{prop:1}
Let $E\to\Omega$ be a trivial holomorphic vector bundle over a domain $\Omega$ in $\C^n$, and let $h$ be a singular hermitian metric on $E$, that is strictly negatively curved in the sense of Nakano, as in Definition \ref{df:3}. Let furthermore $\{h_\nu\}$ be an approximating sequence of smooth metrics, decreasing pointwise to $h$ on any relatively compact subset of $\Omega$, obtained through convolution of $h$ with an approximate identity.

If $h$ is continuous, then for every $\nu$, $h_\nu$ is also strictly negatively curved in the sense of Nakano, with the same constant $\delta$ in (\ref{eq:str_naka}).
\end{prop}

The continuity of $h$ is needed to make sure that the approximating sequence $\{h_\nu\}$, obtained through convolution, converges uniformly to $h$, which will turn out to be important in the proof.

In section 5 we start by giving a similar definition and approximation result in the simpler Griffiths setting. As the dual of a Griffiths negative bundle is Griffiths positive, the regularisation can also be made in the positive case. In \cite{R} this is used to prove a Demailly-Nadel type of vanishing theorem for vector bundles over complex curves (\cite{R} Theorem 1.3).

Unlike curvature in the sense of Griffiths, the dual of a Nakano negative bundle, in general, is not Nakano positive. Because of this we can not use Definition \ref{df:3} and our regularisation result in the positive setting. Hence for singular metrics a whole new approach to Nakano positivity is probably needed. Unfortunately we have so far failed to come up with an appropriate definition that lends itself well to regularisations.

\section*{Acknowledgments}
\noindent It is a pleasure to thank Bo Berndtsson for inspiring and helpful discussions. I would also like to thank Mark Andrea A. de Cataldo.

\section{Curvature and Positivity}
\noindent Let $X$ be a complex manifold with dim$_\C X=n$, and let $(E,h)$ be a smooth, hermitian, holomorphic vector bundle over $X$ with rank $E=r$. Then we have a well-defined bilinear form which we denote by $\langle,\rangle$, for forms on $X$ with values in $E$ by letting $\langle\alpha\otimes s,\beta\otimes t\rangle:=\alpha\wedge\bar{\beta}\ (s,t)_h$ for forms $\alpha,\beta$ and sections $s,t$, and then extend to arbitrary forms with values in $E$ by linearity.

In the main part of this article we are in a local setting and so we will assume that $X$ is a polydisk $U$ and that $E$ is trivial. Hence we regard $h$ as a matrix-valued function on $U$. If $s$ and $t$ are sections of $E$ we will regard them as vectors of functions so that
$$(s,t)_{h}=t^*hs$$
where $t^*$ is the transpose conjugate of $t$ and juxtaposition denotes matrix multiplication.

We denote the Chern connection associated to the bilinear form $\langle,\rangle$ by $D=D'+\dbar$ and the curvature by $\Theta=D^2=D'\dbar+\dbar D'$. Locally $D'$, and hence $D$, can be represented by a matrix of one-forms $\theta$, and one can show that this matrix is given by $\theta=h^{-1}\partial h$. One can also show that $\Theta$ equals $\dbar\theta=\dbar(h^{-1}\partial h)$, i.e. the curvature can locally be represented as a matrix of two-forms. Thus if we let $\{z_1,\ldots,z_n\}$ denote local coordinates on $X$, we have that
\be\label{eq:curv}
\Theta=\sum_{j,k=1}^n\Theta_{jk}dz_j\wedge d\bar{z}_k
\ee
where $\Theta_{jk}$ are $r\times r$ matrix-valued functions on $X$.

Throughout this article, unless explicitly stated otherwise, we will use the $\sim$ symbol to denote form-valued objects acting on some arbitrary, smooth vector field $\xi$. For example, $\tilde{\theta}=\theta(\xi)$, $\tilde{\partial h}=\partial h(\xi)$ and so on. However, for the curvature tensor we let $\tilde{\Theta}:=i\Theta(\xi,\xi)$. 

Now there are two main notions of positivity for holomorphic vector bundles. We say that $(E,h)$ is strictly positively curved in the sense of Griffiths if for some $\delta>0$
$$\big(\tilde{\Theta}s,s\big)_h\geq\delta\|s\|^2_h|\xi|^2$$
for any section $s$ of $E$, and any smooth vector field $\xi$. Using (\ref{eq:curv}) we see that this is equivalent to
$$\sum_{j,k=1}^n\big(\Theta_{jk}s,s\big)_h\xi_j\bar\xi_k\geq\delta\|s\|^2_h|\xi|^2$$
for any vector $\xi\in\C^n$.

We say that $E$ is strictly positively curved in the sense of Nakano if for some $\delta>0$
\be\label{eq:nak}
\sum_{j,k=1}^n\big(\Theta_{jk}s_j,s_k\big)_h\geq\delta\sum_{j=1}^n\|s_j\|^2_h
\ee
for any $n$-tuple $(s_1,\ldots,s_n)$ of sections of $E$. Griffiths and Nakano semipositivity, seminegativity and strict negativity are defined similarly.

Choosing $s_j=\xi_js$ in (\ref{eq:nak}) it is immediate that being positively or negatively curved in the sense of Nakano implies being positively or negatively curved in the sense of Griffiths. The converse however, does not hold in general. Of these two main positivity concepts Griffiths positivity has the nicest functorial properties in that if $E$ is positively curved in the sense of Griffiths, then the dual bundle $E^*$ has negative Griffiths curvature. This property will be very useful for us as it allows us to study metrics with positive and negative Griffiths curvature interchangeably. Unfortunately this correspondence between $E$ and $E^*$ does not hold in the Nakano case. The reason for studying Nakano positivity is that it is intimately related with the solvability of the inhomogeneous $\dbar$-equation using $L^2$ methods.

The smoothness of $h$ is central in defining the curvature, and hence also in being curved in the sense of Griffiths and Nakano. However, as we will be dealing with singular metrics, these definitions will not work for us. Instead the following alternative characterizations will be useful.

Let $u$ be an arbitrary holomorphic section of $E$. Then a short computation yields
\be\label{eq:psh}
i\partial\dbar\|u\|_{h}^2=-\langle i\Theta u,u\rangle_{h}+i\langle D^\prime u,D^\prime u\rangle_{h}\geq-\langle i\Theta u,u\rangle_{h}
\ee
Hence we see that if the curvature is negative in the sense of Griffiths, then $\|u\|^2_{h}$ is plurisubharmonic. On the other hand we can always find a holomorphic section $u$ such that $D^\prime u=0$ at a point. Thus $h$ is negatively curved in the sense of Griffiths if and only if $\|u\|^2_{h}$ is plurisubharmonic, for any holomorphic section $u$.

In exactly the same way one can also show that $h$ is negatively curved in the sense of Griffiths if and only if $\log\|u\|_h^2$ is plurisubharmonic for every holomorphic section $u$. Alternatively, one can obtain this from the well-known fact that for a positive valued function $v$, $\log v$ is plurisubharmonic if and only if $ve^{2\Re q}$ is plurisubharmonic for every polynomial $q$. Choosing $v=\|u\|^2_h$ we get that $\|u\|^2_he^{2\Re q}=\|ue^q\|^2_h$ which is plurisubharmonic as $ue^q$ also is a holomorphic section, for every polynomial $q$.

Now turning to curvature in the sense of Nakano, we let $u=(u_1,\ldots,u_n)$ denote an $n$-tuple of holomorphic sections of $E$ and define $T_u$, an $(n-1,n-1)$-form through
$$T_u=\sum_{j,k=1}^n(u_j,u_k)_h\widehat{dz_j\wedge d\bar z_k}$$
where $(z_1,\ldots,z_n)$ are local coordinates on $X$, and $\widehat{dz_j\wedge d\bar z_k}$ denotes the wedge product of all $dz_i$ and $d\bar z_i$ except $dz_j$ and $d\bar z_k$, multiplied by a constant of absolute value 1, chosen so that $T_u$ is a positive form. Then we have that
{\setlength\arraycolsep{2pt}
\begin{eqnarray}
i\partial\dbar T_u & = & -\sum_{j,k=1}^n\big(\Theta_{jk}u_j,u_k\big)_hdV+\Big\|\sum_{j=1}^nD'_{z_j}u_j\Big\|^2_hdV\geq\nonumber\\
&\geq&-\sum_{j,k=1}^n\big(\Theta_{jk}u_j,u_k\big)_hdV,\nonumber
\end{eqnarray}}
\!\!where $dV:=i^ndz_1\wedge d\bar z_1\wedge\ldots\wedge dz_n\wedge d\bar z_n$ and $D^\prime_{z_j}:=D'_{\partial/\partial z_j}$. Hence analogous to the previous case, we see that if the curvature has negative Nakano curvature, then $T_u$ is a plurisubharmonic $(n-1,n-1)-$form, and on the other hand we can always find holomorphic sections $\{u_j\}$ such that $D_{z_j}'u_j=0$ at a point, for $j=1,\ldots,n$. Thus $T_u$ is a plurisubharmonic $(n-1,n-1)-$form if and only if $h$ is negatively curved in the sense of Nakano.

\section{Comparison with de Cataldo and Theorem \ref{thm:counter}}
\noindent As noted in the introduction, singular hermitian metrics have been introduced and studied previously by Mark Andrea de Cataldo in \cite{C}. However despite the almost identical titel, the purpose and contents of \cite{C} differ quite a lot from ours. In the introduction of \cite{C} it is stated that the main goal is to study global generation problems, (that had previously been studied by Demailly and Siu, among others, in the line bundle case), in the vector bundle setting. For this the 'analytic package' of the higher rank analogues of singular hermitian metrics, regularisation-approximation, $L^2$ estimates, and the Demailly-Nadel vanishing theorem are needed.

However as the main focus in \cite{C} is on algebraic-geometric aspects, and not on the technical regularisation procedures, all such approximation results are taken as part of the definitions. Thus $h$ is defined to be a singular hermitian metric on a vector bundle $E$ over a manifold $X$, if there exists a closed set $\Sigma\subseteq X$ of measure zero and a sequence of (smooth) hermitian metrics $\{h_s\}$ such that $\lim_{s\to\infty}h_s=h$ in the $C^2_{loc}-$topology on $X\setminus\Sigma$. The curvature tensor associated to $h$ is defined to be the curvature tensor of $h$ restricted to $X\setminus\Sigma$; (\cite{C} Defintion 2.1.1).

Immediately after this, (in \cite{C} section 2), this definition is discussed in the line bundle setting. There it is recalled that if $h=h_0e^{-2\varphi}$ is a singular hermitian metric on a line bundle $L$, where $h_0$ is a (smooth) hermitian metric on $L$ and $\varphi$ is a locally integrable function on $X$, then
$$\Theta^h(L)=\Theta^{h_0}(L)+2i\partial\dbar\varphi_{ac}+2i\partial\dbar\varphi_{sing}.$$
Here $2i\partial\dbar\varphi_{ac}$ has locally integrable coefficients and $2i\partial\dbar\varphi_{sing}$ is supported on some closed set $\Sigma$ of measure zero. It is then remarked that in the sense of Definition 2.1.1,
$$\Theta^h(L)=\Theta^{h_0}(L)+2i\partial\dbar\varphi_{ac},$$
i.e. $2i\partial\dbar\varphi_{sing}$ will not be taken into consideration.

For us, finding a current corresponding to $2i\partial\dbar\varphi_{sing}$ in the vector bundle setting has been on of the main motivations behind this work.

Now as Defintion \ref{df:1} is very liberal, more requirements are needed in order to be able to reach any interesting conclusions.

First off, we need our vector bundles to be holomorphic, since we in a holomorphic frame have explicit expressions for the connection and curvature in terms of the metric. Secondly, in practice some sort of positivity condition on the curvature is usually needed. Hence the idea has been to require this from the start, as in Definition \ref{df:2}, and then try to define the curvature tensor as a matrix of positive (or negative) measures. Moreover since $\Theta=\dbar\theta$, for this to work we need that the entries of $\theta$ are locally integrable. However the simple counter example of Theorem \ref{thm:counter} showes that this is not always possible. Let us study this example in more detail.

\begin{proof}[Proof of Theorem \ref{thm:counter}]
Recall that the metric is given by
\begin{displaymath}
h=\left(\begin{array}{cc}
1+|z|^2 & z \\
\bar{z} & |z|^2
\end{array} \right).
\end{displaymath}
One can check that for any holomorphic section $u(z)=\big(u_1(z),u_2(z)\big)$ of $E=\Delta\times\C^2$,
$$\|u\|^2_h=|zu_1(z)|^2+|u_1(z)+zu_2(z)|^2$$
which is subharmonic, so $h$ has negative Griffiths curvature in the sense of Definition \ref{df:2}. 

It is now straightforward to verify that
\begin{displaymath}
\partial h=\left(\begin{array}{cc}
\bar{z} & 1 \\
0 & \bar{z}
\end{array} \right)dz,
\end{displaymath}
and that
\begin{displaymath}
h^{-1}=\frac{1}{|z|^4}\left(\begin{array}{cc}
|z|^2 & -z \\
-\bar{z} & 1+|z|^2
\end{array} \right).
\end{displaymath}
Hence a short calculation yields that the connection matrix corresponding to $h$ is
\begin{displaymath}
\theta^h:=h^{-1}\partial h=\frac{1}{|z|^4}\left(\begin{array}{cc}
\bar{z}|z|^2 & 0 \\
-\bar{z}^2 & \bar{z}|z|^2
\end{array} \right)dz=\left(\begin{array}{cc}
\frac{1}{z} & 0 \\
-\frac{1}{z^2} & \frac{1}{z}
\end{array} \right)dz.
\end{displaymath}
This is clearly not locally integrable on $\Delta$, and furthermore, we see that $\dbar$ of the non-integrable element can be thought of as the derivative of $\delta_0$, which is a distribution of order one. Hence $\Theta^h$ can not be defined in this way as a form valued matrix of measures.
\end{proof}

The conclusion is that at least our approach using holomorphic frames, i.e. trying to define $\Theta^h$ through $\dbar(h^{-1}\partial h)$, simply can not work. It is quite possible that one might be able to achieve this using an orthogonal frame instead, i.e. by trying to define the curvature through $\Theta=d\theta+\theta\wedge\theta$. The problem with this approach is how to make sense of the concept of an orthogonal frame. Recall that in the smooth setting this is just a set of sections that are linearly independent at every point and orthogonal with respect to the metric, but for a singular metric $h$, it is not clear what this means on the singular locus of $h$. For this reason we had to impose the extra condition $\det h>\varepsilon$ for some $\varepsilon>0$, as the singular locus of $h$ is characterised by the fact that $\det h$ vanishes there.

Now as previously mentioned, a Demailly-Nadel type of vanishing theorem in the vector bundle setting is one of the goals in \cite{C}. In order to generalise the line bundle proof (\cite{D1} Th\'eor\`eme 5.1) the notion of strict positivity in the sense of Nakano is needed. Furthermore one also needs to be able to approximate a singular hermitian metric that is strictly Nakano positive, with a sequence of smooth hermitian metrics that are also strictly positively curved in the sense of Nakano.

In \cite{C} this is basically taken as part of the definition, i.e. a hermitian metric $h$ which is singular in the sense of Definition 2.1.1 is said to be strictly positively curved if there exists a sequence of strictly positively curved curvature tensors approximating $\Theta^h$; (\cite{C} Definition 2.4.1). (This is a simplification. The actual definition is more technical, but, as stated in the beginning of section 2.4, the idea is to incorporate the requirements needed to obtain $L^2$-estimates-type results into the definition.)

Clearly this also is quite different from our approach as one of our main goals has been to try to come up with definitions of being strictly positively or strictly negatively curved in the sense of Griffiths and Nakano, in the singular setting, that are possible to regularise without this being part of the definitions. 

\section{Basic Properties}
\noindent In this section we prove Proposition \ref{prop:basic}(ii) and \ref{prop:conn}. (Part (i) can be found in \cite{BP}, Proposition 3.1. We also prove it in the first part of the proof of Proposition \ref{prop:4.1} below.) Recall that we use $\sim$ to denote form-valued objects acting on some arbitrary, smooth vector field $\xi$. For example, $\tilde{\theta}^{h_\nu}=\theta^{h_\nu}(\xi)$, $\tilde{\partial h_\nu}=\partial h_\nu(\xi)$ and so on. Also $\tilde{\Theta}^{h_\nu}:=i\Theta^{h_\nu}(\xi,\xi)$.

\begin{proof}[Proof of Proposition \ref{prop:basic}(ii)] From part (i) of the proposition we know that on any polydisc, there exists a sequence of smooth hermitian metrics $\{h_\nu\}$ on $E$, with negative Griffiths curvature, decreasing pointwise to $h$ on any smaller polydisc. This sequence induces a sequence of metrics $\{\det h_\nu\}$ on the line bundle $\det E$. Furthermore the curvature of the induced metric $\det h_\nu$ is the trace of the corresponding curvature $\Theta^{h_\nu}$, hence negative. Since $\det h_\nu$ is a negatively curved metric on a line bundle, we know that $\det h_\nu=e^{\varphi_\nu}$ where $\varphi_\nu$ is plurisubharmonic, i.e. $\log\det h_\nu$ is a plurisubharmonic function.

As $\{h_\nu\}$ is a decreasing sequence, $\{\det h_\nu\}$ will be decreasing as well. One way to see this is to regard $\det h_\nu$ as a quotient of volumes through the change of variables formula for integrals. Another way is to use the fact that given any two hermitian matrices it is always possible to find a basis in which both matrices are diagonal. Hence $\{\log\det h_\nu\}$ is a decreasing sequence of plurisubharmonic functions and then it is a well known fact that the limit function, $\log\det h$, will be plurisubharmonic as well, (or identically  equal to minus infinity).
\end{proof}

\begin{remark}\label{remark:uniTr}
From this proof we also get that if $\det h\not\equiv0$, then $tr(\tilde{\Theta}^{h_\nu})\in L^1_{loc}(X)$ uniformly in $\nu$. This follows from the observation that the curvature of the metric $\det h_\nu$ is the trace of $\Theta^{h_\nu}$, and the well-known fact that $i\partial\dbar$ of a plurisubharmonic function is a locally finite measure.\qed
\end{remark}

We now turn to the proof of Proposition \ref{prop:conn}. Let $h$ denote a singular hermitian metric with negative Griffiths curvature in the sense of Definition \ref{df:2}, and let $\{h_{\nu}\}$ be an approximating sequence. Just as in Theorem \ref{thm:main} we do not assume that $h_\nu$ is the convolution of $h$ with an approximate identity, but merely that $\{h_\nu\}$ is any sequence of smooth hermitian metrics with negative Griffiths curvature decreasing pointwise to $h$.

Since each $\tilde{\Theta}^{h_\nu}$ is hermitian with respect to $h_\nu$, there is a basis of eigenvectors $\{v_j\}$ that are orthonormal with respect to $h_\nu$. Hence expanding any section $s$ of $E$ in terms of this basis we get that
$$\big(\tilde{\Theta}^{h_\nu}s,s\big)_{h_\nu}=\sum_{j,k=1}^n\big(\lambda_ja_jv_j,a_kv_k\big)_{h_\nu}=\sum_{j=1}^n\lambda_j|a_j|^2.$$ 
Moreover, each $h_\nu$ being negatively curved in the sense of Griffiths by definition means that $\tilde{\Theta}^{h_\nu}$ is negative definite with respect to $\{h_\nu\}$. Thus the eigenvalues $\{\lambda_j\}$ are all negative. In particular we have that
\be\label{eq:herm2}
\big|\big(\tilde{\Theta}^{h_\nu}s,s\big)_{h_\nu}\big|\leq\max_{j=1,\ldots,n}|\lambda_j|\|s\|^2_{h_\nu}\leq-tr(\tilde{\Theta}^{h_\nu})\|s\|^2_{h_\nu}.
\ee
This observation will be of importance in the proof of Proposition \ref{prop:conn} below. Before we turn to this proof, however, we begin with the following simple lemma, that will turn out to be useful as well.

\begin{lma}\label{lma:wconv}
Let $\{f_\nu\}$ be a sequence of functions converging weakly in the sense of distributions to a function $f$. If $f_\nu\in L^2_{loc}$ uniformly in $\nu$, then in fact $f_\nu$ converges weakly in $L^2$ to $f$.
\end{lma}

\begin{proof}
Let $\phi\in L^2_{loc}$. We want to show that
$$\int\phi f_\nu\to\int\phi f\qquad\textrm{as}\quad\nu\to\infty,$$
and we know that this holds if $\phi\in C^\infty_c$.

By taking the convolution of $\phi$ with an approximate identity, we get a sequence $\{\phi_\mu\}$ of smooth functions of compact support, converging to $\phi$ in $L^2$. Furthermore, as $f_\nu\in L^2_{loc}$ uniformly in $\nu$ and the $L^2$-norm of a weakly convergent sequence decreases, we have that $f\in L^2_{loc}$ as well. Thus by the Cauchy-Schwartz inequality
{\setlength\arraycolsep{2pt}
\begin{eqnarray}
|\int\phi(f_\nu-f)|&\leq&|\int\phi_\mu(f_\nu-f)|+\|\phi-\phi_\mu\|_{L^2}\|f_\nu-f\|_{L^2}\leq\nonumber\\
&\leq&|\int\phi_\mu(f_\nu-f)|+C\|\phi-\phi_\mu\|_{L^2}.\nonumber
\end{eqnarray}}
\!\!Taking the limit first in $\nu$, and then in $\mu$, finishes the proof of the lemma.
\end{proof}

\begin{proof}[Proof of Proposition \ref{prop:conn}]
As observed in the Remark \ref{remark:uniTr}, $tr(\tilde{\Theta}^{h_\nu})\in L^1_{loc}(X)$ uniformly in $\nu$. Thus from (\ref{eq:herm2}) we get that $(\tilde{\Theta}^{h_\nu}u,u)_{h_\nu}\in L^1_{loc}(X)$ uniformly in $\nu$ as well.

If $u$ is assumed to be holomorphic, (\ref{eq:psh}) yields
$$i\partial\dbar\|u\|^2_{h_\nu}(\xi,\xi)=-(\tilde{\Theta}^{h_\nu}u,u)_{h_\nu}+\|\tilde{D}'_{h_\nu}u\|^2_{h_\nu}.$$
For any test form $\phi$ of bidegree $(n-1,n-1)$ partial integration gives
$$\int\phi\wedge i\partial\dbar\|u\|^2_{h_\nu}=\int\|u\|^2_{h_\nu}i\partial\dbar\phi.$$
As $\|u\|^2_{h_\nu}$ decreases pointwise to $\|u\|^2_h$, which in turn is assumed to be plurisubharmonic, we get that $i\partial\dbar\|u\|^2_{h_\nu}(\xi,\xi)\in L^1_{loc}(X)$ uniformly in $\nu$. Thus $\|\tilde{D}'_{h_\nu}u\|^2_{h_\nu}\in L^1_{loc}(X)$ uniformly in $\nu$ as well.

Now for constant $u$,
$$\|\tilde{D}'_{h_\nu}u\|^2_{h_\nu}=(\tilde{\theta}^{h_\nu}u)^*h_\nu(\tilde{\theta}^{h_\nu}u)=u^*(\tilde{\partial h}_\nu)^*h_\nu^{-1}(\tilde{\partial h}_\nu)u$$
and so $tr((\tilde{\partial h}_\nu)^*h_\nu^{-1}(\tilde{\partial h}_\nu))\in L^1_{loc}(X)$ uniformly in $\nu$. Since the sequence $\{h_\nu\}$ is locally bounded from above, the sequence $\{h_\nu^{-1}\}$ will be locally bounded from below. Together these facts yield
$$tr\big((\tilde{\partial h}_\nu)^*h_\nu^{-1}(\tilde{\partial h}_\nu)\big)\geq Ctr\big((\tilde{\partial h}_\nu)^*(\tilde{\partial h}_\nu)\big)=C\|\tilde{\partial h}_\nu\|^2_{HS}$$
where the norm denotes the Hilbert-Schmidt norm. One way to see this is to note that as the statement is pointwise, there is no loss of generality in assuming that $h_\nu^{-1}$ is diagonal. Hence $\tilde{\partial h}_\nu\in L^2_{loc}(X)$ uniformly in $\nu$.

For constant $u$ we furthermore have that
$$\partial\|u\|^2_{h_\nu}(\xi)=(\tilde{D}_{h_\nu}'u,u)_{h_\nu}=u^*h_\nu\tilde{\theta}^{h_\nu}u=u^*\tilde{\partial h}_\nu u,$$
and as the left hand side converges weakly in the sense of distributions to $\partial\|u\|^2_h(\xi)$, we can deduce that $\tilde{\partial h}_{\nu}$ converges weakly to $\tilde{\partial h}$ through polarization. Since we moreover just have shown that the $L^2$-norms are bounded, by Lemma \ref{lma:wconv} the convergence is in fact weakly in $L^2$.  

It is a well-known fact that the $L^p$ norm of the limit of a weakly convergent sequence is smaller than the $L^p$ norm of the elements in the sequence, and so these results altogether yield that $\partial h$ is an $L^2$ valued form.

Finally we also know that $\log\det h\in L^1_{loc}(X)$ so that $\det h\neq0$ a.e. Hence $\theta^h:=h^{-1}\partial h$ is well-defined a.e. 
\end{proof}

\section{Proof of Theorem \ref{thm:main}}
\noindent In this section $h$ will always denote a singular hermitian metric with negative Griffiths curvature in the sense of Definition \ref{df:2}, and $\{h_{\nu}\}$ will denote an approximating sequence. Note that in Theorem \ref{thm:main} we are \textit{not} assuming that $h_\nu$ is the convolution of $h$ with an approximate identity, but merely that $\{h_\nu\}$ is any sequence of smooth hermitian metrics with negative Griffiths curvature decreasing pointwise to $h$. Recall that we use $\sim$ to denote form-valued objects acting on some arbitrary, smooth vector field $\xi$. For example, $\tilde{\theta}^{h_\nu}=\theta^{h_\nu}(\xi)$, $\tilde{\partial h_\nu}=\partial h_\nu(\xi)$ and so on. Also $\tilde{\Theta}^{h_\nu}:=i\Theta^{h_\nu}(\xi,\xi)$. Furthermore, as Theorem \ref{thm:main} is local in nature, we will think of $h$, $h_\nu$, $\theta^{h_\nu}$ etc. as matrix-valued functions on some polydisc $U$.

The idea for the proof of Theorem \ref{thm:main} is to use the condition $\det h>\varepsilon$ to show that:\vspace{0.1cm}\\
(i) $\tilde{\theta}^{h_\nu}\in L^2_{loc}(X)$, uniformly in $\nu$.\\
(ii) $\tilde{\theta}^{h_\nu}$ converges weakly to $\tilde{\theta}^h$ in $L^2$.\vspace{0.1cm}\\
(In (ii) we mean that after choosing a basis for $E$ and representing $\tilde{\theta}^{h_\nu}$ as a matrix, each matrix element converges weakly). As the $L^p$ norm of the limit of a weakly convergent sequence is smaller than the $L^p$ norm of the elements in the sequence, this in turn implies that $\tilde{\theta}^h\in L^2_{loc}(X)$ as well. Thus we can deduce that $\Theta^h:=\dbar\theta^h$ is well-defined in the sense of currents, and that the sequence of curvature tensors $\tilde{\Theta}^{h_\nu}$ converges weakly to $\tilde{\Theta}^h$.

The proof of (i) is immediate. Let $\hat{h}$ denote the adjugate of $h$, i.e. $h^{-1}=(\det h)^{-1}\hat{h}$. The assumption $\det h>\varepsilon$ implies that $\det h_\nu>\varepsilon$ and so $h_\nu^{-1}=(\det h_\nu)^{-1}\hat{h}_\nu<\frac{C}{\varepsilon}I$ as the entries of $\hat{h}_\nu$ are just polynomials of the entries of $h_\nu$, i.e. locally bounded from above.

Hence
$$\int\|\tilde{\theta}^{h_\nu}\|^2_{HS}\leq\frac{C}{\varepsilon^2}\int\|\tilde{\partial h}_\nu\|^2_{HS}$$
and so by the proof of Proposition \ref{prop:conn}, $\tilde{\theta}^{h_\nu}\in L^2_{loc}(X)$ uniformly in $\nu$.

Part (ii) is only slightly more difficult. We want to show that for any test function $\chi\in L^2_{loc}(X)$,
$$\int\chi(\tilde{\theta}^{h_\nu}_{jk}-\tilde{\theta}^{h}_{jk})=\int\chi(h_\nu^{-1}\tilde{\partial h}_\nu-h^{-1}\tilde{\partial h})_{jk}\to0\quad\textrm{as}\quad\nu\to\infty.$$
By adding and subtracting the term $(h^{-1}\tilde{\partial h}_\nu)_{jk}$ we get
$$|\int\chi(\tilde{\theta}^{h_\nu}_{jk}-\tilde{\theta}^{h}_{jk})|\leq|\int\chi\big((h_\nu^{-1}-h^{-1})\tilde{\partial h}_\nu\big)_{jk}|+|\int\chi\big(h^{-1}(\tilde{\partial h}_\nu-\tilde{\partial h})\big)_{jk}|.$$
Now the first term converges to zero since by the Cauchy-Schwartz inequality
$$|\int\chi\big((h_\nu^{-1}-h^{-1})\tilde{\partial h}_\nu\big)_{jk}|^2\leq C\int_K\|h_\nu^{-1}-h^{-1}\|^2_{HS}\int_K\|\tilde{\partial h}_\nu\|^2_{HS},$$ 
where $K$ denotes a compact subset of $X$. We know, from Proposition \ref{prop:conn}, that the second factor is bounded uniformly in $\nu$. Furthermore, as previously noted, the assumption $\det h>\varepsilon$ makes $h_\nu^{-1}$ bounded from above. Hence the first factor converges to zero by the dominated convergence theorem. 

For the second term we note that as $h^{-1}$ is bounded from above, we have in particular that $h^{-1}\in L^2_{loc}(X)$. From the proof of Proposition \ref{prop:conn} we know that $\tilde{\partial h_\nu}$ converges weakly in $L^2$ to $\tilde{\partial h}$. Thus the second term goes to zero as well, and we are done.\qed\vspace{0.5cm}

We include a proof of the fact that if we already know that $\tilde{\Theta}^{h_\nu}\in L^1_{loc}(X)$ uniformly in $\nu$, then we can prove Theorem \ref{thm:main} without using the assumption $\det h>\varepsilon$.

The argument becomes much more involved in this case, although the main idea is the same as above. For simplicity we will only treat the one dimensional case. The several variable version is similar but requires a more advanced integral representation formula.

It turns out that (i) is a little too much to hope for. Instead we will replace it with:\vspace{0.1cm}\\
(i)': $\tilde{\theta}^{h_\nu}\in L^p_{loc}(X)$, uniformly in $\nu$, for $1<p<2$.

Let $\Delta$ denote a disc in $\C$ and assume that $E=\Delta\times\C^r$. After choosing a basis for $E$ we can represent $\tilde{\theta}^{h_\nu}$ and $\tilde{\Theta}^{h_\nu}$ as matrices, and we will denote different elements of these matrices by $\tilde{\theta}^{h_\nu}_{jk}$ and $\tilde{\Theta}^{h_\nu}_{jk}$. In this notation, Cauchy's formula yields that
$$\tilde{\theta}^{h_\nu}_{jk}(z)=C\!\!\int\frac{\chi(\zeta)}{\zeta-z}\tilde{\Theta}^{h_\nu}_{jk}(\zeta)d\lambda(\zeta)+C'\!\!\int\frac{\dbar\chi(\zeta)\land\tilde{\theta}^{h_\nu}_{jk}(\zeta)}{\zeta-z}=:f^\nu_{jk}(z)+g^\nu_{jk}(z)$$
for some bump function $\chi$ which we choose such that $\chi\equiv1$ in a neighborhood of $z$ so that $g^\nu_{jk}$ is holomorphic on $\Delta$. 

Now by Jensen's inequality, for any compact subset $K$ of $\Delta$
$$\int_K|f^\nu_{jk}(z)|^p\leq C\int_K\left(\int\frac{\chi^p}{|\zeta-z|^p}\tilde{\Theta}^{h_\nu}_{jk}(\zeta)d\lambda(\zeta)\right)$$
which is integrable for $1<p<2$ uniformly in $\nu$ since $|\zeta-z|^{-p}\in L^p_{loc}(\Delta)$ for $1<p<2$. 

For $g^{\nu}_{jk}$ we proceed in two steps. First we assume that rank$E=1$ so that $h_\nu$ is not matrix-valued. Then
$$h_\nu f^\nu+h_\nu g^\nu=h_\nu\tilde{\theta}^{h_\nu}=\tilde{\partial h}_\nu\in L^2_{loc}(\Delta)$$
and since $h_\nu$ is locally bounded from above, $h_\nu f^\nu\in L^p_{loc}(\Delta)$ and so $h_\nu g^\nu\in L^p_{loc}(\Delta)$ uniformly in $\nu$ for $1<p<2$. By Jensen's inequality once again
$$\exp\left(C\int_K\log|h_\nu g^\nu|\right)\leq C\int_K|h_\nu g^\nu|$$
so
$$\int_K\log|h_\nu g^\nu|<\infty.$$
Since $\log|h_\nu g^\nu|=\log h_\nu+\log|g^\nu|$ and we know that $\log h_\nu$ is subharmonic, we get that
$$\int_K\log|g^\nu|<\infty$$
uniformly in $\nu$. However we also know that $g^\nu$ is holomorphic so $\log|g^\nu|$ is also subharmonic. Thus by the sub-mean inequality $g^\nu$ is bounded and so if rank$E=1$, $\tilde{\theta}^{h_\nu}\in L^p_{loc}(\Delta)$ uniformly in $\nu$ for $1<p<2$.

For general, matrix-valued $h_\nu$ it follows as in the one dimensional case that $h_\nu g^\nu\in L^p_{loc}(\Delta)$ uniformly in $\nu$ for $1<p<2$. Let $\hat{h}_\nu$ denote the adjugate of $h_\nu$ so that $h_\nu^{-1}=(\det h_\nu)^{-1}\hat{h}_\nu$. Since $h_\nu$ is locally bounded from above and the entries of $\hat{h}_\nu$ are polynomials of the entries of $h_\nu$
$$(\det h_\nu)g^\nu=\hat{h}_\nu h_\nu g^\nu\in L^p_{loc}(\Delta).$$
For each entry of the matrix $(\det h_\nu)g^\nu$ it follows from Jensens inequality just as in the one dimensional case that
$$\int_K\log|(\det h_\nu)g^\nu_{jk}|<\infty$$
and since as we have seen $\log|\det h_\nu|=\log\det h_\nu$ is plurisubharmonic, hence locally integrable, $\tilde{\theta}^{h_\nu}\in L^p_{loc}(\Delta)$ uniformly in $\nu$ for $1<p<2$. This proves (i)'.

Finally, let $A_\varepsilon:=\{\det h>\varepsilon\}$. Then by H\"older's inequality, for $\chi\in C^\infty_c(X)$
\begin{eqnarray}
|\!\int\!\!\chi\tilde{\theta}^{h_\nu}_{jk}|&\!\!\leq&|\!\int_{A_\varepsilon}\!\!\chi\tilde{\theta}^{h_\nu}_{jk}|+\int_{A^c_\varepsilon}\!\!|\chi\tilde{\theta}^{h_\nu}_{jk}|\leq\nonumber\\
&\leq&|\!\int_{A_\varepsilon}\!\!\chi\tilde{\theta}^{h_\nu}_{jk}|+C\left(\int_{A^c_\varepsilon}\!\!|\tilde{\theta}^{h_\nu}_{jk}|^p\right)^{1/p}\!\!\!|A^c_\varepsilon|^{\frac{p}{p-1}}\leq|\!\int_{A_\varepsilon}\!\!\chi\tilde{\theta}^{h_\nu}_{jk}|+C|A^c_\varepsilon|^{\frac{p}{p-1}}\nonumber
\end{eqnarray}
where $1<p<2$. We already know that $\log\det h\in L_{loc}^1(X)$, and so $|A^c_\varepsilon|\to0$ as $\varepsilon\to0$. Hence it is enough to prove convergence on $A_\varepsilon$. But then we are in the original setting of Theorem \ref{thm:main}, which we proved above.

\section{Approximation results}
\noindent In the introduction we defined what it means for a metric to be strictly negatively curved in the sense of Nakano and pointed out that the usefulness of Definition \ref{df:3} stems from the fact that it lends itself well to such regularisations. This is the content of Proposition \ref{prop:1} and the aim of this section is to prove this result.

However, in order to illustrate the main ideas more clearly, we begin by proving a similar result in the simpler Griffiths setting. Hence we start by defining what it means for a singular hermitian metric to be strictly curved in the sense of Griffiths. (In what follows we will assume that $F=\{\det h=0\}$ is a closed set, and that there exists an exhaustion of open sets $\{U_j\}$ of $F^c$ such that $\det h>1/j$ on $U_j$.)

\begin{df}\label{df:4}
We say that a singular hermitian metric $h$ on a holomorphic vector bundle $E$ is strictly negatively curved in the sense of Griffiths if:\\
(i) $h$ is negatively curved in the sense of Definition \ref{df:2}. In particular, by Theorem \ref{thm:main}, $\Theta^h$ exists as a current on $F^c$.\\
(ii) there exists some $\delta>0$ such that on $F^c$
\be\label{eq:GrStr}
\sum_{j,k=1}^n\big(\Theta_{jk}^hs,s\big)_h\xi_j\bar{\xi}_k\leq-\delta\|s\|^2_h|\xi|^2
\ee
in the sense of distributions, for any section $s$ and any vector $\xi\in\C^n$.

We say that $h$ is strictly positively curved in the sense of Griffiths, if the corresponding dual metric is strictly negatively curved.
\end{df}

We now have the following approximation result.

\begin{prop}\label{prop:4.1}
Let $h$ be a singular hermitian metric on a trivial holomorphic vector bundle $E$, over a domain $\Omega$ in $\C^n$, and assume that $h$ is strictly negatively curved in the sense of Griffiths, as in Definition \ref{df:4}. Let furthermore $\{h_\nu\}$ be an approximating sequence of smooth metrics, decreasing pointwise to $h$ on any relatively compact subset of $\Omega$, obtained through convolution of $h$ with an approximate identity.

If $h$ is continuous, then for every $\nu$, $h_\nu$ is also strictly negatively curved in the sense of Griffiths, with the same constant $\delta$ in (\ref{eq:GrStr}).
\end{prop}

\begin{proof}
Let $h_\nu=h\ast\chi_\nu$ where $\chi_\nu$ is an approximate identity, i.e. $\chi\in C_c^\infty(\Omega)$ with $\chi\geq0$, $\chi(p)=\chi(|p|)$, $\int\chi=1$, and $\chi_\nu(p)=\nu^n\chi(\nu p)$. We start by showing that $\{h_\nu\}$ is decreasing in $\nu$, and that for every holomorphic section $u$ and any $\nu$, $\|u\|_{h_\nu}^2$ is plurisubharmonic.

By definition, $\{h_\nu\}$ is a decreasing sequence if for any constant section $s$, $\|s\|_{h_\nu}^2$ is decrasing. However if $s$ is constant then $\|s\|_{h}^2$ is plurisubharmonic and it is immediate from the definition of $h_\nu$ that
$$\|s\|_{h_\nu}^2=\|s\|_{h}^2\ast\chi_\nu.$$
The statement now follows from the well known fact that convolutions of plurisubharmonic functions are decreasing.

For the second statement we note that for any holomorphic section $u$
$$\|u\|_{h_\nu}^2(p)=\int K(p,q)\chi_\nu(q)dV(q)$$
where (locally)
$$K(p,q)=u^*(p)h(p-q)u(p).$$
Since $\|u\|^2_h$ is assumed to be plurisubharmonic for any holomorphic section $u$, for fixed $q$, $K(p,q)$ will be plurisubharmonic, (and locally bounded from above), in $p$. As furthermore $\chi_\nu dV$ is a positive measure of compact support the result follows from another well-known property of plurisubharmonic functions.

From section 2 we know that for the smooth metrics $h_\nu$
\be\label{eq:4.2}
i\partial\dbar\|u\|^2_{h_\nu}\geq-\langle i\Theta^{h_\nu}u,u\rangle_{h_\nu}
\ee
for any holomorphic section $u$. Our next goal is to show that this inequality holds for $h$ in the sense of distributions.

Now for any test form $\phi$ of bidegree $(n-1,n-1)$
$$\int\phi\wedge i\partial\dbar\|u\|^2_{h_\nu}=\int\|u\|^2_{h_\nu}i\partial\dbar\phi$$
and by the monotone convergence theorem, the right hand side converges. Hence $i\partial\dbar\|u\|^2_{h_\nu}$ converges weakly to $i\partial\dbar\|u\|^2_h$ as measures.

For the right hand side of (\ref{eq:4.2}) we express the curvature tensor in terms of a local basis
$$\Theta^{h_\nu}=\sum_{j,k=1}^n\Theta_{jk}^{h_\nu}dz_j\wedge d\bar{z}_k$$
and let $\phi=\phi_{jk}\widehat{idz_j\wedge d\bar z_k}$, where $\phi_{jk}$ is supported on $F^c$ and $\widehat{idz_j\wedge d\bar z_k}$ denotes the wedge product of all $dz_i$ and $d\bar z_i$ except $dz_j$ and $d\bar z_k$, multiplied by a constant of absolute value 1, chosen so that $idz_j\wedge d\bar z_k\wedge\widehat{idz_j\wedge d\bar z_k}=dV$. We then have that
{\setlength\arraycolsep{2pt}
\begin{eqnarray}
&&\int\phi\wedge\Big(\langle i\Theta^{h_\nu}u,u\rangle_{h_\nu}-\langle i\Theta^hu,u\rangle_h\Big)=\!\!\int\phi_{jk}\ u^*\Big(h_{\nu}\Theta_{jk}^{h_\nu}-h\Theta_{jk}^h\Big)udV\leq\nonumber\\
&&\quad\leq\big|\int\phi_{jk}\ u^*(h_\nu-h)\Theta_{jk}^{h_\nu}u\ dV\big|+\big|\int\phi_{jk}\ u^*h\big(\Theta_{jk}^{h_\nu}-\Theta_{jk}^h\big)u\ dV\big|.\nonumber
\end{eqnarray}}
\!\!We know that $h_\nu$ converges uniformly to $h$ since, by assumption, $h$ is continuous. As $\Theta^{h_\nu}\in L_{loc}^1(F^c)$ uniformly in $\nu$, this implies that the first term converges to zero on $F^c$. Furthermore, by Theorem \ref{thm:main} we know that $\Theta^{h_\nu}$ converges weakly to $\Theta^h$ as currents with measurable coefficients, and so the second term also converges to zero on $F^c$. Thus we have that $h_\nu\Theta^{h_\nu}$ converges weakly to $h\Theta^h$ on $F^c$.

From (\ref{eq:4.2}) and Definition \ref{df:4} we now get that
$$i\partial\dbar\|u\|^2_{h}\geq-\langle i\Theta^{h}u,u\rangle_{h}\geq\delta\|u\|^2_h\sum_{j=1}^nidz_j\wedge d\bar{z}_j$$
in the sense of distributions on $F^c$. Moreover, by Definition \ref{df:4} the left hand side is a positive measure. Hence if we let $\chi_{F^c}$ and $\chi_F$ denote the characteristic functions of $F^c$ and $F$ respectively, we get
{\setlength\arraycolsep{2pt}
\begin{eqnarray}
i\partial\dbar\|u\|^2_{h}&=&\chi_{F^c}i\partial\dbar\|u\|^2_{h}+\chi_{F}i\partial\dbar\|u\|^2_{h}\geq\nonumber\\
&\geq&\chi_{F^c}i\partial\dbar\|u\|^2_{h}\geq\chi_{F^c}\delta\|u\|^2_h\sum_{j=1}^nidz_j\wedge d\bar{z}_j.\nonumber
\end{eqnarray}}
\!\!Since furthermore $F$ is a set of Lebesgue measure zero, we altogether have that
\be\label{eq:4.3}
i\partial\dbar\|u\|^2_{h}\geq\delta\|u\|^2_h\sum_{j=1}^nidz_j\wedge d\bar{z}_j
\ee
in the sense of distributions (on all of $\Omega$).

If we let $h^q(p):=h(p-q)$ we have
$$\|u\|_{h_\nu}^2(p)=\int\|u\|_{h^q}^2(p)\chi_\nu(q)dV(q).$$
Combining this with (\ref{eq:4.3}) yields
{\setlength\arraycolsep{2pt}
\begin{eqnarray}
&&i\partial\dbar \|u\|^2_{h_\nu}(p)=\int\chi_\nu(q)i\partial\dbar\|u\|^2_{h^q}(p)dV(q)\geq\nonumber\\
&&\quad\geq\delta\int\|u\|^2_{h^q}(p)\chi_\nu(q)dV(q)\sum_{j}^nidz_j\wedge d\bar{z}_j=\delta\|u\|^2_{h_\nu}(p)\sum_{j}^nidz_j\wedge d\bar{z}_j\nonumber.
\end{eqnarray}}

Lastly, as discussed in section 2, as long as the metric is smooth, i.e. for fix $\nu$, the section $u$ can always be chosen so that
$$i\partial\dbar \|u\|^2_{h_\nu}=-\langle i\Theta^{h_\nu}u,u\rangle_{h_\nu}.$$
Hence
$$\sum_{j,k=1}^n\big(\Theta_{jk}^{h_\nu}u,u\big)_{h_\nu}idz_j\wedge d\bar{z}_k\leq-\delta\|u\|^2_{h_\nu}\sum_{j=1}^nidz_j\wedge d\bar{z}_j$$
which is what we wanted to prove.
\end{proof}

\begin{remark}
As the dual of a strictly Griffiths negative metric is strictly Griffiths positive, a corresponding approximation result holds for singular hermitian metrics that are strictly positively curved in the sense of Griffiths, as in Definition \ref{df:4}. Thus for vector bundles over Riemann surfaces, where the notions of Griffiths and Nakano curvature coincide, one can use Proposition \ref{prop:4.1} to prove a Demailly-Nadel type of vanishing theorem (see \cite{R} Theorem 1.2).\qed
\end{remark}

We now turn to the proof of Proposition \ref{prop:1}. As already mentioned it is very similar to the previous proof, and so we will be rather sketchy.

\begin{proof}[Proof of Proposition \ref{prop:1}] Let $h_\nu:=h\ast\chi_\nu$ where $\chi_\nu$ is an approximate identity. We begin by observing that if $u=(u_1,\ldots,u_n)$ is an $n$-tuple of holomorphic sections of $E$, then locally
$$(u_j,u_k)_{h_\nu}(p)=\int(u_j,u_k)_{h^q}(p)\chi_\nu(q)dV(q),$$
where $h^q(p):=h(p-q)$ just as before. In particular if we let $p=(z_1,\ldots,z_n)$ and study the $(n-1,n-1)$-form
$$T_u^h=\sum_{j,k=1}^n(u_j,u_k)_h\widehat{dz_j\wedge d\bar z_k}$$
introduced and discussed in sections 1 and 2, we see that
\be\label{eq:Tappr}
T_u^{h_\nu}(p)=\int T_u^{h^q}(p)\chi_\nu(q)dV(q).
\ee
By Definition \ref{df:3}, $h$ is negatively curved in the sense of Nakano if $T^h_u$ is plurisubharmonic for any holomorphic $n-$tuple $u$. Thus in exactly the same way as in the previous argument we get that $T_u^{h_\nu}$ is a plurisubharmonic $(n-1,n-1)-$form, and that $T^{h_\nu}_u$ decrease pointwise to $T_{u}^h$.

Now from section 2 we know that for smooth metrics
\be\label{eq:ineq}
i\partial\dbar T_u^{h_\nu}\geq -\sum_{j,k=1}^n\big(\Theta_{jk}^{h_\nu}u_j,u_k\big)_{h_\nu}dV.
\ee
It follows from the same argument as in the previous proof that this inequality still holds on $F^c$, in the sense of distributions, in the singular setting. Together with Definition \ref{df:3} this yields that on $F^c$
$$i\partial\dbar T_{u}^{h}\geq -\sum_{j,k=1}^n\big(\Theta_{jk}^{h}u_j,u_k\big)_{h}dV\geq\delta\sum_{j=1}^n\|u_j\|^2_hdV.$$
From Definition \ref{df:3} we also know that the left hand side is a positive measure, and so by the same reasoning as before we have
$$i\partial\dbar T_{u}^{h}\geq\delta\sum_{j=1}^n\|u_j\|^2_hdV$$
in the sense of distributions (on all of $\Omega$).

Combined with (\ref{eq:Tappr}) this in turn gives
{\setlength\arraycolsep{2pt}
\begin{eqnarray}
i\partial\dbar T_u^{h_\nu}(p)&=&\int\chi_\nu(q)i\partial\dbar T_{u}^{h^q}(p)dV(q)\geq\nonumber\\
&\geq&\delta\sum_{j,k=1}^n\Big(\int\|u_j\|_{h^q}^2(p)\chi_\nu(q)dV(q)\Big)dV(p)=\nonumber\\
&=&\delta\Big(\sum_{j=1}^n\|u_j\|^2_{h_\nu}(p)\Big)dV(p).\nonumber
\end{eqnarray}}

Finally, from section 2 we know that as long as the metric is smooth, the sections $u$ can always be chosen so that
$$i\partial\dbar T_u^{h_\nu}=-\sum_{j,k=1}^n\big(\Theta_{jk}^{h_\nu}u_j,u_k\big)_{h_\nu}dV.$$
Hence
$$\sum_{j,k=1}^n\big(\Theta_{jk}^{h_\nu}u_j,u_k\big)_{h_\nu}\leq-\delta\sum_{j=1}^n\|u_j\|^2_{h_\nu}$$
and we are done.
\end{proof}

\begin{remark}
For Nakano positive metrics this argument will not work since we do not have any counterpart of inequality (\ref{eq:ineq}) in that setting. In the proof we did use the fact that we can always choose the $n$-tuple $u$ such that equality holds in (\ref{eq:ineq}), but this was for a fix $\nu$, and it is not possible to do this in such a way that equality holds \textit{uniformly} in $\nu$.

In the Griffiths setting the positive case was addressed just by studying duals. As mentioned in the introduction, the dual of a Nakano negative bundle in general is not Nakano positive, and so this approach is not possible here. However for dual Nakano negative bundles, the proposition applies.\qed
\end{remark}


\begin{bibdiv}
\begin{biblist}



\bib{B}{article}{
   author={Berndtsson, Bo},
   title={Curvature of vector bundles associated to holomorphic fibrations},
   journal={Ann. of Math. (2)},
   volume={169},
   date={2009},
   number={2},
   pages={531--560},
}

\bib{BP}{article}{
   author={Berndtsson, Bo},
   author={P{\u{a}}un, Mihai},
   title={Bergman kernels and the pseudoeffectivity of relative canonical
   bundles},
   journal={Duke Math. J.},
   volume={145},
   date={2008},
   number={2},
   pages={341--378},
}

\bib{C}{article}{
   author={de Cataldo, Mark Andrea A.},
   title={Singular Hermitian metrics on vector bundles},
   journal={J. Reine Angew. Math.},
   volume={502},
   date={1998},
   pages={93--122},
}

\bib{D1}{article}{
   author={Demailly, Jean-Pierre},
   title={Estimations $L^{2}$ pour l'op\'erateur $\bar \partial $ d'un
   fibr\'e vectoriel holomorphe semi-positif au-dessus d'une vari\'et\'e
   k\"ahl\'erienne compl\`ete},
   language={French},
   journal={Ann. Sci. \'Ecole Norm. Sup. (4)},
   volume={15},
   date={1982},
   number={3},
   pages={457--511},
}

\bib{D2}{article}{
   author={Demailly, Jean-Pierre},
   title={Singular Hermitian metrics on positive line bundles},
   conference={
      title={Complex algebraic varieties},
      address={Bayreuth},
      date={1990},
   },
   book={
      series={Lecture Notes in Math.},
      volume={1507},
      publisher={Springer},
      place={Berlin},
   },
   date={1992},
   pages={87--104},
}

\bib{N}{article}{
   author={Nadel, Alan Michael},
   title={Multiplier ideal sheaves and K\"ahler-Einstein metrics of positive
   scalar curvature},
   journal={Ann. of Math. (2)},
   volume={132},
   date={1990},
   number={3},
   pages={549--596},
}

\bib{R}{article}{
   author={Raufi, Hossein},
   title={The Nakano vanishing theorem and a vanishing theorem of Demailly-Nadel type for vector bundles},
   date={2012},
   status={Preprint},
   eprint={arXiv:1212.4417 [math.CV]},
   url={http://arxiv.org/abs/1212.4417}

}






 
\end{biblist}
\end{bibdiv}
\end{document}